\documentclass[reqno]{amsart}
\usepackage[margin=1.5in,bottom=1.25in]{geometry}		

\usepackage{amsmath}		
\usepackage{amssymb}		
\usepackage{amsfonts}		
\usepackage{amsthm}		
\usepackage[foot]{amsaddr}		

\usepackage{mathtools}		

\mathtoolsset{%
}

\usepackage[utf8]{inputenc}		
\usepackage[T1]{fontenc}		

\usepackage[
cal=cm,
]
{mathalfa}

\usepackage{dsfont}		

\usepackage[proportional,tabular,lining,sf,mono=false]{libertine}

\usepackage{acronym}		
\newcommand{\acli}[1]{\emph{\acl{#1}}}		
\newcommand{\acdef}[1]{\emph{\acl{#1}} \textup{(\acs{#1})}\acused{#1}}		

\usepackage[labelfont={bf,small},labelsep=colon,font=small]{caption}	
\usepackage{subcaption}		
\usepackage[svgnames]{xcolor}	
\colorlet{MyRed}{DarkRed!50!Crimson}
\colorlet{MyBlue}{DodgerBlue!75!black}
\colorlet{MyGreen}{DarkGreen}
\colorlet{MyViolet}{DarkMagenta}

\colorlet{MyLightBlue}{DodgerBlue!20}
\colorlet{MyLightGreen}{MyGreen!20}

\colorlet{PrimalColor}{MyBlue}
\colorlet{PrimalFill}{MyLightBlue}
\colorlet{DualColor}{MyRed}

\colorlet{AlertColor}{MyRed}	
\colorlet{BadColor}{MyRed}	
\colorlet{GoodColor}{MyGreen}	
\colorlet{LinkColor}{MediumBlue}	
\colorlet{MacroColor}{MyRed}
\colorlet{RevColor}{MediumBlue}	

\newcommand{\afterhead}{.\;}		
\newcommand{\para}[1]{\medskip\paragraph{\textbf{#1\afterhead}}}

\usepackage{array}		
\usepackage{booktabs}		
\usepackage[inline,shortlabels]{enumitem}		
\setenumerate{itemsep=\smallskipamount,topsep=2pt,left=\parindent}
\setitemize{itemsep=\smallskipamount,topsep=2pt,left=\parindent}

\usepackage[kerning=true]{microtype}		

\usepackage{comment}
\usepackage{tabto}		
\usepackage{xspace}		

\usepackage[sort&compress]{natbib}		

\bibpunct[, ]{[}{]}{,}{n}{,}{,}

\usepackage{hyperref}
\hypersetup{
final,
colorlinks=true,
linktocpage=true,
pdfstartview=FitH,
breaklinks=true,
pdfpagemode=UseNone,
pageanchor=true,
pdfpagemode=UseOutlines,
plainpages=false,
bookmarksnumbered,
bookmarksopen=false,
bookmarksopenlevel=1,
hypertexnames=false,
pdfhighlight=/O,
urlcolor=LinkColor,linkcolor=LinkColor,citecolor=LinkColor,	
pdftitle={},
pdfauthor={},
pdfsubject={},
pdfkeywords={},
pdfcreator={pdfLaTeX},
pdfproducer={LaTeX with hyperref}
}

\newcommand{\EMAIL}[1]{\email{\href{mailto:#1}{#1}}}

\usepackage[sort&compress,capitalize,nameinlink]{cleveref}		
\crefname{algo}{Algorithm}{Algorithms}
\crefname{assumption}{Assumption}{Assumptions}
\crefname{figure}{Fig.}{Figs.}
\crefname{model}{Model}{Models}

\usepackage{algorithm}		
\usepackage{algpseudocode}		

\usepackage{thmtools}		
\usepackage{thm-restate}		

\theoremstyle{plain}
\newtheorem{theorem}{Theorem}		
\newtheorem{lemma}{Lemma}		
\newtheorem{proposition}{Proposition}		

\newtheorem*{corollary*}{Corollary}		

\theoremstyle{definition}
\newtheorem{definition}{Definition}		

\newtheorem*{example*}{Example}		

\theoremstyle{remark}

\newtheorem*{remark*}{Remark}

\usepackage[showdeletions]{color-edits}		
\newcommand{\draft}[1]{#1}		

\newcommand{\newmacro}[2]{\newcommand{#1}{\draft{#2}}}		
\newcommand{\newop}[2]{\DeclareMathOperator{#1}{\draft{#2}}}		

\DeclarePairedDelimiter{\bracks}{[}{]}		
\DeclarePairedDelimiter{\parens}{(}{)}		

\DeclarePairedDelimiter{\pospart}{[}{]_{+}}		

\DeclarePairedDelimiterX{\setof}[1]{\{}{\}}{#1}		
\DeclarePairedDelimiterX{\setdef}[2]{\{}{\}}{#1:#2}		
\DeclarePairedDelimiterXPP{\exclude}[1]{\mathopen{}\setminus}{\{}{\}}{}{#1}

\newcommand{\R}{\mathbb{R}}		

\DeclareMathOperator{\bigoh}{\mathcal{O}}		
\DeclareMathOperator{\dist}{dist}		
\newop{\simplex}{\Delta}		

\newcommand{\cf}{cf.\xspace}		
\newcommand{\eg}{e.g.,\xspace}		
\newcommand{\ie}{i.e.,\xspace}		

\newcommand{\textpar}[1]{\textup(#1\textup)}		

\newcommand{\alt}[1]{#1'}		
\newcommand{\altalt}[1]{#1''}		

\newmacro{\dd}{\:d}		

\newmacro{\const}{c}		
\newmacro{\Const}{C}		
\newmacro{\coef}{\lambda}		

\newmacro{\param}{\theta}		
\newmacro{\params}{\Theta}		

\newmacro{\tstart}{0}		
\renewcommand{\time}{\draft{t}}		
\newmacro{\timealt}{s}		
\newmacro{\horizon}{T}		

\newmacro{\flow}{\phi}		
\DeclarePairedDelimiterXPP{\flowof}[2]{\flow_{#1}}{(}{)}{}{#2}		

\newop{\Nash}{NE}		
\newop{\brep}{br}		
\newop{\reg}{Reg}		
\newop{\val}{val}		

\newmacro{\players}{\mathcal{I}}		

\newmacro{\pure}{i}		
\newmacro{\purealt}{j}		
\newmacro{\purealtalt}{k}		
\newmacro{\nPures}{m}		
\newmacro{\pures}{\mathcal{A}}		

\newmacro{\strat}{x}		
\newmacro{\straty}{y}
\newmacro{\stratcl}{c_l}	
\newmacro{\stratcr}{c_r}	
\newmacro{\stratalt}{\alt\strat}		
\newmacro{\strataltalt}{\altalt\strat}		
\newmacro{\strats}{\mathcal{X}}		
\newmacro{\intstrats}{\strats^{\circle}}		

\newmacro{\eq}{p}		

\newmacro{\pay}{u}		
\newmacro{\payv}{v}		
\newmacro{\pot}{\Phi}		

\newmacro{\game}{\mathcal{G}}		
\newmacro{\gamefull}{\game(\pures,\payv)}		

\newmacro{\fingame}{\Gamma}		
\newmacro{\fingamefull}{\Gamma(\players,\pures,\pay)}		
\newmacro{\mixgame}{\simplex(\fingame)}

\newmacro{\mat}{M}		
\newmacro{\hmat}{H}		

\newmacro{\ones}{\mathbf{1}}		
\newmacro{\eye}{I}		
\newmacro{\zer}{\mathbf{0}}		
\newmacro{\ttop}{{\!\top\!}}		

\newcommand{\defeq}{\coloneqq}		

\newcommand{\from}{\colon}		

\newmacro{\beforestart}{0}		
\newmacro{\start}{1}		
\newmacro{\afterstart}{2}		
\newmacro{\running}{\start,\afterstart,\dotsc}		

\newmacro{\run}{n}		
\newmacro{\runalt}{k}		
\newmacro{\nRuns}{T}		

\newmacro{\state}{\strat}		
\newmacro{\statealt}{\score}		

\newcommand{\curr}[1][\state]{\draft{#1}_{\run}}		
\renewcommand{\next}[1][\state]{\draft{#1}_{\run+1}}		

\addauthor[Fryderyk]{FF}{DarkViolet}

\renewcommand{\phi}{\varphi}
\newmacro{\fbio}{\bio}
\newmacro{\fecon}{\econ}
\newmacro{\fcs}{\learn}

\addauthor[Pan]{PM}{MediumBlue}

\newop{\RD}{RD}

\newmacro{\updmap}{f}
\newmacro{\bio}{\cramped{\updmap_{\mathclap{\step}}^{\mathrm{I}}}}
\newmacro{\econ}{\cramped{\updmap_{\mathclap{\step}}^{\mathrm{II}}}}
\newmacro{\learn}{\cramped{\updmap_{\mathclap{\step}}^{\mathrm{III}}}}
\newmacro{\econp}{\cramped{\updmap_{\mathclap{\; \step,p}}^{\mathrm{II}}}}
\newmacro{\econpp}{\cramped{\updmap_{\mathclap{\; \step,1-p}}^{\mathrm{II}}}}
\newmacro{\point}{x}
\newmacro{\set}{\mathcal{S}}
\newmacro{\size}{z}
\newmacro{\score}{y}
\newmacro{\step}{\delta}
\newmacro{\switch}{\rho}
\newmacro{\gain}{\gamma}
\newmacro{\payvec}{\pi}

\begin{document}


\title
[Discrete-time replicator dynamics: convergence, instability, and chaos]
{On the discrete-time origins of the replicator dynamics:\\
From convergence to instability and chaos}

\author
[F.~Falniowski]
{Fryderyk Falniowski$^{\ast}$}
\address{$^{\ast}$\,%
Department of Mathematics, Krakow University of Economics, Rakowicka 27, 31-510 Krak\'{o}w, Poland.}
\EMAIL{falniowf@uek.krakow.pl}
\author
[P.~Mertikopoulos]
{Panayotis Mertikopoulos$^{\sharp}$}
\address{$^{\sharp}$\,%
Univ. Grenoble Alpes, CNRS, Inria, Grenoble INP, LIG, 38000 Grenoble, France.}
\EMAIL{panayotis.mertikopoulos@imag.fr}



\newacro{LHS}{left-hand side}
\newacro{RHS}{right-hand side}
\newacro{iid}[i.i.d.]{independent and identically distributed}
\newacro{lsc}[l.s.c.]{lower semi-continuous}
\newacro{whp}[w.h.p]{with high probability}
\newacro{wp1}[w.p.$1$]{with probability $1$}
\newacro{ODE}{ordinary differential equation}

\newacro{CCE}{coarse correlated equilibrium}
\newacroplural{CCE}[CCE]{coarse correlated equilibria}
\newacro{NE}{Nash equilibrium}
\newacroplural{NE}[NE]{Nash equilibria}
\newacro{ESS}{evolutionarily stable state}

\newacro{RD}{replicator dynamics}
\newacro{MWU}{multiplicative weights update}
\newacro{PPI}{pairwise proportional imitation}
\newacro{EW}{exponential\,/\,multiplicative weights}
\newacro{EXP3}{exponential-weights algorithm for exploration and exploitation}

\newacro{GAN}{generative adversarial network}

\begin{abstract}
We consider three distinct discrete-time models of learning and evolution in games:
a biological model based on intra-species selective pressure,
the dynamics induced by \acl{PPI},
and
the \acl{EW} algorithm for online learning.
Even though these models share the same continuous-time limit \textendash\ the \acl{RD} \textendash\ we show that second-order effects play a crucial role and may lead to drastically different behaviors in each model, even in very simple, symmetric $2\times2$ games.
Specifically, we study the resulting discrete-time dynamics in a class of parametrized congestion games, and we show that
\begin{enumerate*}
[\upshape(\itshape i\hspace*{.5pt}\upshape)]
\item
in the biological model of intra-species competition, the dynamics remain convergent for any parameter value;
\item
the dynamics of \acl{PPI} for different equilibrium configurations exhibit an entire range of behaviors for large step size (stability, instability, and even Li-Yorke chaos);
while
\item
for the \ac{EW} algorithm increasing step size will (almost) inevitably 
lead to chaos (again, in the formal, Li-Yorke sense).
\end{enumerate*}
This divergence of behaviors comes in stark contrast to the globally convergent behavior of the \acl{RD}, and serves to delineate the extent to which the \acl{RD} provide a useful predictor for the long-run behavior of their discrete-time origins.
\end{abstract}
\maketitle

\allowdisplaybreaks		

\section{Introduction}
\label{sec:introduction}

Ever since it was introduced by \citet{Nas50}, the notion of a \acl{NE} and its refinements have remained among the most prominent solution concepts of noncooperative game theory.
As such, one of the most fundamental questions in the field has been to specify whether \textendash\ and under what conditions \textendash\ players eventually end up emulating an equilibrium (or equilibrium-like) behavior through repeated interactions;
put differently, whether a dynamic process driven by the agents' individual interests converges to a rational outcome,
in which (classes of) games, etc.

Historically, this question fueled the intense interest in game dynamics brought about by the inception of evolutionary game theory in the mid-1970's,
then the surge of activity that followed in the field of economic theory in the 1990's,
and, more recently, through various connections to machine learning and artificial intelligence, in theoretical computer science. 
Accordingly, depending on the context, game dynamics are usually derived in one of the following ways:
\begin{enumerate*}
[\upshape(\itshape i\hspace*{.5pt}\upshape)]
\item
from a biological model of population evolution, typically phrased in terms of the reproductive fitness of the species involved;
\item
from a set of economic microfoundations that express the growth rate of a type (or strategy) within a population via a \emph{revision protocol} (an economic model prescribing an agent's propensity to switch to a better-performing strategy);
or
\item
from some learning algorithm designed to optimize a myopic performance criterion (such as the minimization of an agent's regret), in an otherwise agnostic setting where the players do not know the game being played.
\end{enumerate*}
This has in turn generated an immense body of literature,
see \eg Hofbauer and Sigmund \cite{HS98} for the biological viewpoint,
\citet{Wei95,FL98} and \citet{San10} for a more economic-oriented approach,
and
\citet{CBL06} for the algorithmic\,/\,information-theoretic aspects of the theory.

Now, depending on the precise context, the question of whether the players' behavior converges to equilibrium or not may admit a wide range of answers, from positive to negative.
Starting with the positive, a folk result states that if the players of a finite game follow a no-regret learning process, the players' empirical frequency of play converges in the long run to the set of \acp{CCE} \textendash\ also known as the game's \emph{Hannan set} \citep{Han57}.
This result has been pivotal for the development of the field because no-regret play can be achieved through fairly simple myopic processes like the \acf{EW} update scheme \citep{Vov90,LW94,ACBFS95,ACBFS02} and its many variants \citep{RS13-NIPS,SS11,Sor09}.
On the downside however
\begin{enumerate*}
[\upshape(\itshape a\upshape)]
\item
this convergence result does not concern the actual strategies employed by the players on a day-by-day basis;
and
\item
in many games, \acp{CCE} may violate even the weakest axioms of rationalizability.
\end{enumerate*}
For example, as was shown by \citet{VZ13}, it is possible for players to have \emph{negative regret} for all time, but nonetheless play \emph{only strictly dominated strategies} for the entire horizon of play.

This takes us to the negative end of the spectrum.
If we focus on the evolution of the players' mixed strategies, a series of well-known impossibility results by \citet{HMC03,HMC06} have shown that there are no uncoupled learning dynamics \textendash\ deterministic or stochastic, in either continuous or discrete time \textendash\ that converge to \ac{NE} in \emph{all} games from any initial condition.%
\footnote{The adjective ``uncoupled'' means here that a player's update rule does not explicitly depend on the other players' strategies (except implicitly, through the player's payoff function).}
In turn, this leads further weight to examining the question of equilibrium convergence within a specific class of games, and for a specific (class of) game dynamics.

In this regard, one of the most \textendash\ if not \emph{the} most \textendash\ widely studied game dynamics are the \acdef{RD} of \citet{TJ78}, arguably the \emph{spiritus movens} of evolutionary game theory.
Originally derived as a model for the evolution of biological populations under selective pressure in the spirit of \citet{Mor62}, the \acl{RD} were subsequently rederived in economic theory via a mechanism known as \acdef{PPI}, originally due to \citet{Hel92},%
\footnote{See also \citet{BinSam97} for a derivation via a related mechanism known as ``imitation driven by dissatisfaction'', complementing the ``imitation of success'';
for a comprehensive account, \cf \citet{San10}.}
and, at around the same time,
as the mean dynamics of a stimulus-response model known as the \acdef{EW} algorithm, \cf \citet{Vov90,LW94,ACBFS95} and \citet{Rus99,Rus99b}.

This convergence of viewpoints is quite remarkable:
even though the starting point of these considerations is a set of conceptually very different and mathematically disparate discrete-time models, they all share the \acl{RD} as a continuous-time limit.
In this way, by studying the \acl{RD}, one can hope to obtain plausible predictions for the long-run behavior of the above models,
at least when the time step $\step$ of the underlying discrete-time model is sufficiently small to justify the descent to continuous time.
However, since real-life modeling considerations often involve larger values of $\step$ (\eg as in the case of species with longer evolutionary cycles or agents with faster revision rates),
we are led to the following natural question:
\medskip
\begin{quote}
\centering
\itshape
Do the discrete-time models underlying the \acl{RD} lead to qualitatively different outcomes?
And, if so, to what extent?
\end{quote}

\para{Our contributions}

One could plausibly expect that the answer to this question is most likely positive in large, complicated games, with a wide range of different behaviors emerging in the long run;
on the other hand, in smaller, simpler games, the range of behaviors that arise would probably be qualitatively similar, and only differ at a quantitative level (such as the rate of convergence to an equilibrium or the like).

Somewhat surprisingly, we show that this expectation is too optimistic, even in the class of potential games (which arguably possess the strongest convergence guarantees under the \acl{RD}), and even for cases where agents are symmetric and only have two actions at their disposal (the smallest meaningful game).
In particular, we consider the case of symmetric random matching in a $2\times2$ congestion game, and we show that the different discrete-time origins of the \acl{RD} exhibit the following qualitatively different behaviors:
\begin{enumerate}
\item
In the biological model of intra-species competition, the dynamics converge to \acl{NE} for any value of the time step $\step>0$.
\item
In the economic model of \acl{PPI},
there exist certain equilibrium configurations that are globally attracting for \emph{any} value of $\step$,
others for which the game's equilibrium is repelling for a range of values of $\step$,
and
yet others that lead, through the loss of equilibrium stability and period doubling, to the emergence of Li-Yorke chaos (for a different range of values of $\step$).
\item
Finally, in the case of the \ac{EW} algorithm, all equilibrium configurations become unstable and, unless gains from abandoning the most congested choice are equal,  Li-Yorke chaos emerges whenever the time-step exceeds a certain threshold depending on the exact position of the game's equilibrium.
\end{enumerate}

In the above, the notion of \emph{Li-Yorke chaos} \textendash\ as introduced in the seminal paper of \citet{liyorke} \textendash\ means that there exists an uncountable set of initial conditions that is \emph{scrambled}, \ie every pair of points in this set eventually comes arbitrarily close and then drifts apart again infinitely often.
In the type of systems that we consider here, Li-Yorke chaos implies other features of chaotic behavior like positive topological entropy or the existence of a set on which one can detect sensitive dependence on initial conditions Ã  la Devaney \cite{Ruette}.
In this sense, the system is truly unpredictable, which comes in stark contrast to the universally convergent landscape that arises in the continuous-time limit of the process (and which is only shared by the biological model above).

In this regard, the discrete-time origins of the \acl{RD} exemplify the mantra ``\emph{discretization matters}'' to the extreme:
The result of descending from discrete to continuous time and back is radically different, even in cases where the underlying continuous-time dynamics exhibt a universally convergent landscape that would make all asymptotic pseudotrajectories of the process (stochastic or deterministic) converge \citep{Ben99}.
We find this outcome particularly intriguing, as it provides a concrete, quantitative cautionary tale for the extent to which the \acl{RD} can serve as a meaningful predictor for the long-run behavior of their discrete-time origins.

\para{Related work}

There is a significant corpus of recent works suggesting that complex, non-equilibrium behaviors of boundedly rational agents (employing learning rules) seems to be common rather than exceptional.
In this aspect, the seminal work of \citet{SatoFarmer_PNAS} showed analytically that even in a simple two-player zero-sum game of Rock-Paper-Scissors, the (symmetric) \acl{RD} exhibit Hamiltonian chaos.
\citet{sato2003coupled} subsequently extended this result to more general multiagent systems, opening the door to detecting chaos in many other games (always in the continuous-time regime).

More recently, \citet{becker2007dynamics} and \citet{geller2010microdynamics} exhibited chaotic behavior for Nash maps in games like matching pennies, while \citet{VANSTRIEN2008259} and \citet{VANSTRIEN2011262} showed that fictitious play also possesses rich periodic and chaotic behavior in a class of 3x3 games, including
Shapley's game and zero-sum dynamics.
In a similar vein, \citet{Soda14} showed that  the replicator dynamics are PoincarÃ© recurrent in zero-sum games, a result which was subsequently generalized to the so-called ``follow-the-regularized-leader'' (FTRL) dynamics \cite{mertikopoulos2017cycles}, even in more general classes of games \cite{LMB23-CDC};
see also \cite{MV22,HMC21,BM23,MHC24} for a range of results exhibiting convergence to limit cycles and other non-trivial attractors.

It is also known that when FTRL is applied with constant step-size in zero-sum games it becomes unstable and in fact Lyapunov chaotic~\cite{CP2019}, while there is growing evidence that a class of algorithms from behavioral game theory known as experience-weighted attraction (EWA) also exhibits chaotic behavior for two-agent games with many strategies in a large parameter space \cite{GallaFarmer_PNAS2013}, or in games with many agents \cite{GallaFarmer_ScientificReport18}.
In particular, \citet{2017arXiv170109043P} showed experimentally that EWA leads to limit cycles and high-dimensional chaos in two-agent games with negatively correlated payoffs.
All in all, careful examination suggests a complex behavioral landscape in many games (small or large) for which no single theoretical framework currently applies.

However, none of the above results implies chaos in the formal sense of Li-Yorke.
The first formal proof of Li-Yorke chaos was shown for the \ac{EW} algorithm in a single instance of two-agent two-strategy congestion game by \citet{palaiopanos2017multiplicative}.
This result was generalized and strengthened (in the sense of positive topological entropy) for all two-agent
two-strategy congestion games~\cite{Thip18}.
In arguably the main precursor of our work \cite{CFMP2019} topological chaos in nonatomic congestion game where agents use EW was established.
This result was then extended to FTRL with steep regularizers~\cite{BCFKMP21} and EWA algorithms \cite{bielawski2024memory}, but the resulting framework does not apply to the range of models from biology and economic theory considered here (species competition and revision protocols respectively).

\para{Notation}

In what follows, we will write $v \cdot w = \sum_{i=1}^{m} v_{i} w_{i}$ for the (Euclidean) inner product between two real $m$-dimensional vectors $v,w\in\R^{m}$, and $v \odot w = (v_{1}w_{1},\dotsc,v_{m}w_{m})$ for their Hadamard product.
Finally, to simplify notation, given a function $f\from\R\to\R$, we will thread it over all elements of $v\in\R^{m}$ by writing $f(v) \defeq (f(v_{1}),\dotsc,f(v_{m}))$.

\section{Preliminaries}
\label{sec:prelims}

Throughout our paper, we will focus on games with a continuum of nonatomic players modeled by the unit interval $\players = [0, 1]$, with each player choosing (in a measurable way) an \emph{action} \textendash\ or \emph{pure strategy} \textendash\ from a finite set indexed by $\pure \in \pures \equiv \setof{1,\dotsc,\nPures}$.
Letting $\strat_{\pure} \in [0,1]$ denote the mass of agents playing $\pure\in\pures$, the overall distribution of actions at any point in time will be specified by the \emph{state of the population} $\strat = (\strat_{1},\dotsc,\strat_{\nPures})$, viewed here as a point in the unit simplex $\strats \defeq \simplex(\pures) = \setdef*{\strat\in\R_{+}^{\nPures}}{\sum_{\pure\in\pures} \strat_\pure =1}$ of $\R^{\nPures}$.
The players' payoffs in this setting are described by an
ensemble of payoff functions $\pay_{\pure}\from\strats\to\R_{+}$ (assumed Lipschitz),
with $\pay_{\pure}(\strat)$ denoting the payoff to agents playing $\pure\in\pures$ when the population is at state $\strat$.
Putting everything together, we will write
$\pay(\strat) = \sum_{\pure\in\pures} \strat_{\pure} \pay_{\pure}(\strat)$ for the population's \emph{mean payoff} at state $\strat\in\strats$,
$\payv(\strat) = (\pay_{1}(\strat),\dotsc,\pay_{\nPures}(\strat))$ for the associated \emph{payoff vector} at state $\strat$,
and
we will refer to the tuple $\game \equiv \gamefull$ as a \emph{population game}.%
\footnote{We focus here on games played by a single population of nonatomic agents.
The extension of our analysis to multi-population settings requires more elaborate notation, but is otherwise straightforward.} Finally, a state $\strat\in\strats$ is a Nash equilibrium of the game $\game$ if every strategy in use earns a maximal payoff (equivalently, each agent in population chooses an optimal strategy with respect to the choices of others).

In the general context of population games, the most widely studied evolutionary game dynamics are the \acdef{RD} of \citet{TJ78}. These are succinctly described by the continuous-time system
\begin{equation}
\label{eq:RD}
\tag{RD}
\dot\strat_{\pure}
	= \strat_{\pure}
		\bracks*{\pay_{\pure}(\strat) - \pay(\strat)}
\end{equation}
which specifies that the per capita growth rate of the population share of a given strategy $\pure\in\pures$ is proportional to the difference between the payoff $\pay_{\pure}(\strat)$ of said strategy and the mean population payoff $\pay(\strat) = \strat \cdot \payv(\strat)=\sum_{\pure=1}^{\nPures}\strat_{\pure}\pay_{\pure}(\strat)$.
For an introduction to the vast literature surrounding the replicator dynamics, see \cite{HS98,Wei95,San10} and references therein.

A specific class of population games \textendash\ and, indeed, the class that will be of most interest to us \textendash\ is obtained when two individuals are selected randomly from the population and are matched to play a symmetric two-player game with payoff matrix $\mat = (\mat_{\pure\purealt})_{\pure,\purealt\in\pures}$.
In this case, the payoff to agents playing $\pure\in\pures$ at state $\strat$ is $\pay_{\pure}(\strat) = \sum_{\purealt\in\pures} \mat_{\pure\purealt} \strat_{\purealt}$,
so the game's payoff field can be written in concise form as $\payv(\strat) = \mat\strat$.
Following standard conventions in the field, we will refer to this scenario as \emph{symmetric random matching} \cite{HS98,Wei95,San10,HLMS22}.

\section{Dynamics}
\label{sec:dynamics}

In this section, we discuss and derive three established models for the evolution of large populations in discrete time.
All three models share the same continuous-time limit, namely the replicator equation \eqref{eq:RD};
however, as we show in \cref{sec:results}, the behavior of each model is drastically different in discrete time, even in the simplest of games.

Most of the material presented in this section is not new, but we chose to include all relevant details for self-completeness and uniformity of notation.

\subsection*{Model I: Intra-species competition}

We begin with the biological microfoundations of an evolutionary model for intra-species competition in the spirit of \citet{Mor62}.%
\footnote{The case of inter-species competition is similar, but the notation is heavier so we do not treat it here.}
Here, each pure strategy $\pure \in \pures = \setof{1,\dotsc,\nPures}$ represents a genotype in a population that reproduces asexually and $\pay_{\pure}(\strat)$ represents the reproductive fitness of the $\pure$-th genotype when the population is at state $\strat \in \strats$ \textendash\ \ie the net number of offspring per capita in the unit of time.
Then, if $\size_{\pure}(\time)$ denotes the \emph{absolute} size of the $\pure$-th genotype at time $\time$ and $\step$ is the interval between generations, the evolution of the population will be governed by the discrete-time dynamics
\begin{equation}
\size_{\pure}(\time+\step)
	= \size_{\pure}(\time)
		+ \size_{\pure}(\time) \pay_{\pure}(\strat(\time)) \step
	\quad
	\textrm{with}
	\quad
\strat_{\pure}(\time)
	= \frac{\size_{\pure}(\time)}{\sum_{\purealt\in\pures} \size_{\purealt}(\time)}.
\end{equation}
Accordingly, expressing everything in terms of population states \textendash\ that is, as a function of the relative frequency $\strat_{\pure}(\time)$ of each genotype \textendash\ we obtain the autonomous dynamics
\begin{align}
\strat_{\pure}(\time+\step)
	&= \frac
		{\size_{\pure}(\time+\step)}
		{\sum_{\purealt\in\pures} \size_{\purealt}(\time+\step)}
	\notag\\
	&= \frac
		{\size_{\pure}(\time) \cdot \bracks{1 + \pay_{\pure}(\strat(\time)) \step}}
		{\sum_{\purealt\in\pures} \size_{\purealt}(\time) \cdot \bracks{1 + \pay_{\purealt}(\strat(\time)) \step}}
	\notag\\
	&= \frac
		{\strat_{\pure}(\time) \cdot \bracks{1 + \pay_{\pure}(\strat(\time)) \step}}
		{1 + \pay(\strat(\time)) \step}
	\label[model]{mod:bio}
	\tag{I}
\end{align}
where, in the last step, we have used the fact that $\sum_{\purealt\in\pures} \strat_{\purealt}(\time) = 1$.

Formally, \cref{mod:bio} is mathematically equivalent to the so-called ``linear multiplicative weights'' algorithm in computer science and learning theory, \cf \citet{Vov90,LW94,AHK12} and references therein.
In addition, for small $\step$, a simple first-order Taylor expansion yields
\begin{align}
\strat_{\pure}(\time+\step) - \strat_{\pure}(\time)
	&= \strat_{\pure}(\time)
		\cdot \bracks{1 + \pay_{\pure}(\strat(\time)) \step}
		\cdot \bracks*{1 - \step \pay(\strat(\time)) + \bigoh(\step)^{2}}
	- \strat_{\pure}(\time)
	\notag\\
	&= \step \strat_{\pure}(\time)
		\bracks*{\pay_{\pure}(\strat(\time)) - \pay(\strat(\time))}
	+ \bigoh(\step^{2})
\end{align}
so, in the continuous-time limit $\step\to0$, we get
\begin{equation}
\dot\strat_{\pure}(\time)
	\sim \frac{\strat_{\pure}(\time+\step) - \strat_{\pure}(\time)}{\step}
	= \strat_{\pure}(\time)
		\bracks*{\pay_{\pure}(\strat(\time)) - \pay(\strat(\time))}
	+ \bigoh(\step).
\end{equation}
In the above, the asymptotic equality sign ``$\sim$'' is to be interpreted loosely and is only meant to suggest that \cref{mod:bio} represents an Euler discretization of \eqref{eq:RD} up to a higher-order $\bigoh(\step^{2})$ correction term.
Because this term is negligible in the continuous-time limit $\step\to0$, \eqref{eq:RD} is commonly regarded in the literature as the mean dynamics of \cref{mod:bio} \citep{HS98,San10}.

\subsection*{Model II: Pairwise proportional imitation}

The second model that we consider has its roots in the mass-action interpretation of game theory and, more precisely, the theory of revision protocols \citep{San10, Wei95}.
Referring to the classical textbook of \citet{San10} for the details, suppose that each agent occasionally receives an opportunity to switch actions \textendash\ say, based on the rings of a Poisson alarm clock \textendash\  and, at such moments, they reconsider their choice of action by comparing its payoff to that of a randomly chosen individual in the population.
A \emph{revision protocol} of this kind is typically defined by specifying the \emph{conditional switch rate} $\switch_{\pure\purealt}(\strat)$ at which a revising $\pure$-strategist switches to strategy $\purealt$ when the population is at state $\strat\in\strats$.
In this case, the population share of agents switching from strategy $\pure$ to strategy $\purealt$ over a small interval of time $\step$ will be be $\strat_{\pure} \switch_{\pure\purealt} \step$, leading to the inflow-outflow equation
\begin{equation}
\label{eq:in-out}
\strat_{\pure}(\time+\step)
	= \strat_{\pure}(\time)
	+ \step \bracks*{
		\sum_{\purealt\neq\pure} \strat_{\purealt}(\time) \switch_{\purealt\pure}(\strat(\time))
		- \strat_{\pure}(\time) \sum_{\purealt\neq\pure} \switch_{\pure\purealt}(\strat(\time))
		}
\end{equation}

One of the most widely studied revision protocols of this type is the \acdef{PPI} of \citet{Hel92}, as described by the switch rate functions
\begin{equation}
\label{eq:PPI}
\tag{PPI}
\switch_{\pure\purealt}(\strat)
	= \strat_{\purealt} \pospart{\pay_{\purealt}(\strat) - \pay_{\pure}(\strat)}.
\end{equation}
Under this protocol, a revising agent first observes the action of a randomly selected opponent, so a $\purealt$-strategist is observed with probability $\strat_{\purealt}$ when the population is at state $\strat\in\strats$.
Then, if the payoff of the incumbent strategy $\pure\in\pures$ is lower than that of the benchmark strategy $\purealt$, the agent imitates the selected agent with probability proportional to the payoff difference $\pospart{\pay_{\purealt}(\strat) - \pay_{\pure}(\strat)}$;
otherwise, the revising agent skips the revision opportunity and sticks to their current action.

Now, substituting the protocol \eqref{eq:PPI} into \eqref{eq:in-out}, a straightforward calculation yields the autonomous, discrete-time dynamics\footnote{Instead of \cref{mod:bio} where $\step$ can be arbitrarily large, in \cref{mod:econ} the step size is bounded as otherwise its iterations may fail to lie in the simplex. }
\begin{align}
\strat_{\pure}(\time+\step)
	&= \strat_{\pure}(\time)
	+ \step \strat_{\pure}(\time)
		\bracks*{
			\sum_{\purealt\neq\pure} \strat_{\purealt}(\time) \pospart{\pay_{\pure}(\strat(\time)) - \pay_{\purealt}(\strat(\time))}
			- \sum_{\purealt\neq\pure} \strat_{\purealt}(\time) \pospart{\pay_{\purealt}(\strat(\time)) - \pay_{\pure}(\strat(\time))}
		}
	\notag\\
	&= \strat_{\pure}(\time)
	+ \step \strat_{\pure}(\time)
		\bracks*{
			\sum_{\purealt\neq\pure} \strat_{\purealt}(\time)
				\bracks{\pay_{\pure}(\strat(\time)) - \pay_{\purealt}(\strat(\time))}
		}
	\notag\\
	&= \strat_{\pure}(\time)
	+ \step \strat_{\pure}(\time)
		\bracks*{\pay_{\pure}(\strat(\time)) - \pay(\strat(\time))}
	\label[model]{mod:econ}
	\tag{II}
\end{align}
where, in the second-to-last line, we used the fact that $\sum_{\purealt\neq\pure} \strat_{\purealt}(\time) = 1-\strat_{\pure}(\time)$. Just like \cref{mod:bio}, \cref{mod:econ} can be seen as an Euler discretization of \eqref{eq:RD};
however, in contrast to its biological counterpart, there is \emph{no} $\bigoh(\step^{2})$ correction term in \cref{mod:econ}.
Albeit negligible in the limit $\step\to0$, we will see in \cref{sec:results} that the residual $\bigoh(\step^{2})$ term that appears in \cref{mod:bio} plays a major role in the long-run behavior of the dynamics.

\subsection*{Model III: Learning with exponential weights}

The last model we consider has its origins in learning theory and, more specifically, the so-called multi-armed bandit problem, \cf \citet{ACBFS95,ACBFS02,CBL06} and references therein.
Following \citet{HLMS22}, this model can be described in our population setting as follows:
at time $\time$, the performance of each strategy $\pure\in\pures$ is scored by measuring its cumulative payoff over an interval of time $\step$;
subsequently, at time $\time+\step$, the population is redistributed with each strategy receiving a population share that is exponentially proportional to its cumulative payoff up to time $\time+\step$ (\ie agents select with exponentially higher probability those strategies that perform better over time).

Formally, this simple stimulus-response model amounts to the \acli{EW} update
\begin{equation}
\label{eq:EW}
\tag{EW}
\score_{\pure}(\time+\step)
	= \score_{\pure}(\time) + \step\pay_{\pure}(\strat(\time))
	\quad
	\text{with}
	\quad
\strat_{\pure}(\time)
	= \frac{\exp(\score_{\pure}(\time))}{\sum_{\purealt\in\pures} \exp(\score_{\purealt}(\time))}
\end{equation}
which, in the context of learning theory, forms the basis of the \acdef{EXP3} \citep{ACBFS95,ACBFS02}.%
\footnote{This particular instantiation of the algorithm is known as Hedge \citep{ACBFS95};
in some other threads of the literature, \eqref{eq:EW} is referred to as the \acdef{MWU} \cite{AHK12}.
We employ the original terminology of \cite{ACBFS95,ACBFS02} to distinguish it from the linearized version of \cite{LW94}.}
Thus, under \eqref{eq:EW}, the associated population shares will be governed by the autonomous dynamics
\begin{align}
\strat_{\pure}(\time+\step)
	&= \frac
		{\exp(\score_{\pure}(\time+\step))}
		{\sum_{\purealt\in\pures} \exp(\score_{\purealt}(\time+\step))}
	\notag\\
	&= \frac
		{\exp(\score_{\pure}(\time)) \exp(\step\pay_{\pure}(\strat(\time)))}
		{\sum_{\purealt\in\pures} \exp(\score_{\purealt}(\time)) \exp(\step\pay_{\purealt}(\strat(\time)))}
	\notag\\
	&= \frac
		{\strat_{\pure}(\time) \exp(\step\pay_{\pure}(\strat(\time)))}
		{\sum_{\purealt\in\pures} \strat_{\purealt}(\time) \exp(\step\pay_{\purealt}(\strat(\time)))}
	\label[model]{mod:learn}
	\tag{III}
\end{align}
where, in the last line, we used the fact that $\strat_{\pure}(\time) \propto \exp(\score_{\pure}(\time))$, as per \eqref{eq:EW}.
A first-order Taylor expansion then yields
\begin{align}
\strat_{\pure}(\time+\step) - \strat_{\pure}(\time)
	&= \strat_{\pure}(\time)
		\bracks*{\frac{\exp(\step\pay_{\pure}(\strat(\time)))}{\sum_{\purealt\in\pures} \strat_{\purealt}(\time) \exp(\step\pay_{\purealt}(\strat(\time)))} - 1}
	\notag\\
	&= \strat_{\pure}(\time)
		\bracks*{\frac{1 + \step\pay_{\pure}(\strat(\time)) + \bigoh(\step^{2})}{1 + \step \pay(\strat(\time)) + \bigoh(\step^{2})} - 1}
	\notag\\
	&= \strat_{\pure}(\time)
		\bracks[\bigg]{
			\parens*{1 + \step\pay_{\pure}(\strat(\time))}
			\parens*{1 - \step \pay(\strat(\time))}
		 + \bigoh(\step^{2})
		- 1}
	\notag\\
	&= \step \strat_{\pure}(\time)
		\bracks*{
			\pay_{\pure}(\strat(\time))
			- \pay(\strat(\time))}
	+ \bigoh(\step^{2})
\end{align}
so \cref{mod:learn} can also be seen as an Euler discretization of \eqref{eq:RD} up to an $\bigoh(\step^{2})$ correction term.
Conceptually, this is similar to \cref{mod:bio}, though the two corrections are, in general, different;
we will see in \cref{sec:results} that this difference plays a major role in the qualitative behavior of dynamics when $\step$ is not infinitesimally small.


\section{Analysis and results}
\label{sec:results}

As we saw, \cref{mod:bio,mod:econ,mod:learn} admit the same continuous-time limit \textendash\ the replicator dynamics \eqref{eq:RD} \textendash\ so it would be natural to expect that they behave similarly in the long run, especially when the time unit $\step$ is infinitesimally small.
However, when real-life modeling considerations call for larger values of $\step$ (\eg as in the case of species with longer evolutionary cycles or agents with faster revision rates), it is not clear if this heuristic holds true (and to what extent), even when \eqref{eq:RD} would suggest a unique long-run outcome.
In view of this, our goal in the sequel will be to examine in detail the asymptotic behavior of \cref{mod:bio,mod:econ,mod:learn}, and to see whether any qualitative differences arise between these models and/or the underlying dynamics \eqref{eq:RD}.

For concreteness, we will focus on the two extremes of the spectrum of possible asymptotic behaviors:
(global) convergence on the one hand, and (deterministic) chaos on the other.
The reason for this is straightforward:
if the dynamics are globally convergent, the population's initial state is eventually forgotten, and all initializations ultimately settle down to the same state;
instead, if the dynamics are chaotic, even arbitrarily small differences in the population's initial state would lead to drastically different behavior. 
As such, convergence and chaos can be seen as antithetical to each other \textendash\ and hence, as opposites in terms of long-run predictions.

Now, to put all this on a solid footing, it will be convenient to recast \cref{mod:bio,mod:econ,mod:learn} in abstract recursive form as
\begin{equation}
\label{eq:dyn}
\next
	= \updmap(\curr)
\end{equation}
where $\curr \defeq \strat(\run\step)$ denotes the state of the population at time $\time = \run\step$, $\run=\running$, and the dynamics' \emph{update map} $\updmap\from\strats\to\strats$ is defined as follows:
\begin{subequations}
\renewcommand{\theequation}{\arabic{parentequation}.\Roman{equation}}
\begin{itemize}
\item
Under \cref{mod:bio}:
\begin{align}
\label{eq:map-bio}
\updmap(\strat)
	&\equiv \frac
		{\strat + \step \, \payv(\strat) \odot \strat}
		{1 + \step \, \pay(\strat)}
\intertext{\item
Under \cref{mod:econ}:}
\label{eq:map-econ}
\updmap(\strat)
	&\equiv \strat
	+ \step \bracks{\payv(\strat) \odot \strat - \pay(\strat) \, \strat}
\intertext{\item
Under \cref{mod:learn}:}
\label{eq:map-learn}
\updmap(\strat)
	&\equiv \frac
		{\strat \odot \exp(\step\payv(\strat))}
		{\strat \cdot \exp(\step\payv(\strat))}
\end{align}
\end{itemize}
\end{subequations}

In this setting we call the fixed point $\strat$ of the map $\updmap$
\begin{itemize}
\item \emph{attracting}, if there is an open neighborhood $U \subset \strats$ of $\strat$ such that for every $y \in U$ we have $\lim\limits_{n \to \infty} \updmap^n(y) = \strat$, where $\updmap^n$ is a composition of the map $f$ with itself $n$-times. 
\item \emph{repelling}, if there is an open neighborhood $U \subset \strats$ of $\strat$ such that for every $y \in U$, $y \neq \strat$ there exists $n \in \mathbb{N}$ such that $\updmap^n(y) \in \strats\backslash U$.
\end{itemize}
As we will study differentiable maps on the unit interval, the fixed point $\strat$ is attracting if $|\updmap'(\strat)| < 1$, and  $\strat$ is repelling when   if $|\updmap'(\strat)| > 1$. If  $|\updmap'(\strat)| =1$ we need more information.
An orbit $\{\updmap^n(\strat)\}$ is called \emph{periodic} of period $T$ if $\updmap^{n+T}(\strat)=\updmap^n(\strat)$ for any $n\in \mathbb{N}$. The smallest such $T$ is called the period of $\strat$.
The periodic orbit is called attracting, if $\strat$ is an attracting fixed point of $(\strats,\updmap^T)$, and
 repelling, if $\strat$ is a repelling fixed point of $(\strats,\updmap^T)$.

On the antipodes of convergence to the fixed point lays chaotic behavior. The most widely used definition of chaos is due to \citet{liyorke}:

\begin{definition}[Li-Yorke chaos]
\label{def:chaos}
Consider a dynamical system of the form \eqref{eq:dyn} for some continuous map $\updmap\from\strats\to\strats$.
A pair of points $\strat,\stratalt\in\strats$ is said to be \emph{scrambled} \textendash\ or a \emph{Li\textendash Yorke pair} \textendash\ if
\begin{equation}
\label{eq:chaos}
\liminf_{\run\to\infty} \dist(\updmap^{\run}(\strat),\updmap^{\run}(\stratalt))
	= 0
	\quad
	\text{and}
	\quad
\limsup_{\run\to\infty} \dist(\updmap^{\run}(\strat),\updmap^{\run}(\stratalt))
	> 0.
\end{equation}
We then say that \eqref{eq:dyn} is \emph{chaotic} \textpar{in the Li\textendash Yorke sense} if it admits an uncountable \emph{scrambled set}, \ie a set $\set\subseteq\strats$ such that every pair of distinct points $\strat,\stratalt\in\set$ is scrambled.
\end{definition}

\begin{remark*}
The origins of \cref{def:chaos} can be traced back to the seminal paper of \citet{liyorke}, where the notion of a scrambled set was introduced as a surrogate for mixing:
intuitively, any two (distinct) orbits starting in a scrambled set will come arbitrarily close to each other and subsequently spring aside infinitely many times.
In one-dimensional systems, which we study here, Li-Yorke chaos implies other features of chaotic behavior like positive topological entropy or existence of the set on which one can detect sensitive dependence on initial conditions in the sense of Devaney \cite{Ruette}.
\end{remark*}

In general, showing that a system exhibits chaotic behavior is a task of considerable difficulty, and a satisfactory theory exists only for low-dimensional systems.
On that account, we will focus on a random matching scenario induced by the $2\times2$ symmetric game with actions $\pures = \{A,B\}$ and payoff bimatrix
\begin{equation}
\label{eq:game}
\begin{array}{l|cc}
	&A	&B\\
\hline
A	&(0,0)		&(\gain_{A},\gain_{B})\\
B	&(\gain_{B},\gain_{A})	&(0,0)
\end{array}
\end{equation}
This is a normalized anti-coordination / congestion game where the parameters $\gain_{A}, \gain_{B} > 0$ reflect the benefit of deviating from the most congested choice:
if the entire population plays $A$, an agent would gain $\gain_{B}$ by deviating to $B$;
and, likewise,
if the entire population plays $B$, an agent would gain $\gain_{A}$ by deviating to $A$.

For this game, the \acl{RD} \eqref{eq:RD} boil down to the one-dimensional system
\begin{equation}
\label{eq:RD-game}
\dot\state
	= \state(1-\state) \bracks{\gain_{A}(1-\state) - \gain_{B}\state}
\end{equation}
where, in a slight abuse of notation, we write $\state \equiv \strat_{A}$ for the population share of $A$-strategists.
This system admits three fixed points, $0$, $1$, and $\eq = \gain_{A} / (\gain_{A} + \gain_{B})$, with the following stability properties:%
\footnote{Recall here that a point is \emph{stable} if all trajectories that start sufficiently close to it remain close for all time;
otherwise, the point is called \emph{unstable}.
Moreover, if all nearby orbits converge to the point in question, it is called \emph{attracting};
and if a point is both stable and attracting, it is called \emph{asymptotically stable}.}
\begin{itemize}
\item
\emph{For the fixed points at $0$ and $1$:}
letting $\RD(\state) \defeq \state (1-\state) \bracks{\gain_{A}(1-\state) - \gain_{B}\state}$ denote the \acs{RHS} of \eqref{eq:RD-game}, we trivially get $\RD'(0) = \gain_{A} > 0$ and $\RD'(1) = \gain_{B} > 0$ so, by standard results in dynamical systems theory \cite{Shu87}, they are both linearly unstable under \eqref{eq:RD-game}.
\item
\emph{For the fixed point at $\eq$:}
working as above, we get $\RD'(\eq) = - \gain_{A} \gain_{B} / (\gain_{A} + \gain_{B}) < 0$, so $\eq$ is linearly stable under \eqref{eq:RD-game}.
\end{itemize}
In fact, it is easy to see that $\eq$ is the unique symmetric \acl{NE} of \eqref{eq:game} and, in fact, it is a global \ac{ESS} thereof;
as a result, the dynamics \eqref{eq:RD-game} converge to $\eq$ from every interior initialization $\state(0) \in (0,1)$.

Given this robust, global convergence landscape and the existence of a (global) \acl{ESS}, one might expect that \cref{mod:bio,mod:econ,mod:learn} enjoy similar convergence properties.
However, as we show below, the long-run behavior of \cref{mod:bio,mod:econ,mod:learn} can be drastically different, ranging from fully convergent to fully chaotic.

\begin{theorem}
\label{thm:chaos-game}
With notation as above, the dynamics \eqref{eq:dyn} with $\updmap$ given by \cref{eq:map-bio,eq:map-econ,eq:map-learn} exhibit for the game \eqref{eq:game} the following asymptotic behavior:
\begin{itemize}
\item
Under \cref{mod:bio}: trajectories of all points of $(0,1)$ converge to Nash equilibrium $\eq$ for any $\step>0$.
\item
Under \cref{mod:econ}, which is well defined for $\step\leq \step^*=\min\left\{4/\eq^2,4/(1-\eq)^2\right\}$:
\begin{enumerate}
[\upshape(1)]
 \item Trajectories of all points of $(0,1)$ converge to Nash equilibrium $\eq$ for $\step\leq 2/\gain_A\gain_B$.
 \item If $\eq \in (0,1/3)\cup (2/3,1)$, then trajectories of all points of $(0,1)$ converge to Nash equilibrium $\eq$ for any $\step$. 
 \item If $\eq \in (1/3,2/3)$, then for $\step\in (2/\gain_A\gain_B,\delta^*)$ the \textpar{unique} Nash equilibrium $\eq$ is repelling and, except for a countable set of initial conditions, all trajectories do not converge to equilibrium.
 \item If $\eq\in (\frac{1}{23}(31-12\sqrt{3}),\frac{1}{23}(12\sqrt{3}-8))$, then there exists a unique $\step_{\eq}^{II}$ such that $\updmap$ has periodic orbits of all periods and is Li-Yorke chaotic for any $\step\in (\step_{\eq}^{II},\step^*)$.
 \end{enumerate}
\item
Under \cref{mod:learn}:
\begin{enumerate}
[\upshape(1)]
\item Trajectories of all points of $(0,1)$ converge to Nash equilibrium $\eq$ for $\step\leq 2/\gain_A\gain_B$.
\item If $\gain_A\neq \gain_B$, then there exists $\step_{\eq}^{III}$ such that if
$\step>\step_{\eq}^{III}$ then $\updmap$ has periodic orbits of all periods and is
Li-Yorke chaotic.
\item If $\gain_A=\gain_B$
and $\step>2/\gain_A^2$, then $\updmap$ has a periodic
attracting orbit $\{\sigma_{\step},1-\sigma_{\step}\}$, where $0<\sigma_{\step}<1/2$. This
orbit attracts trajectories of all points of $(0,1)$, except countably
many points whose trajectories eventually fall into the repelling
fixed point at $1/2$.
\end{enumerate}
\end{itemize}
\end{theorem}

\begin{figure}
\centering
\begin{subfigure}{.5\textwidth}
  \centering
  \includegraphics[width=.9\linewidth]{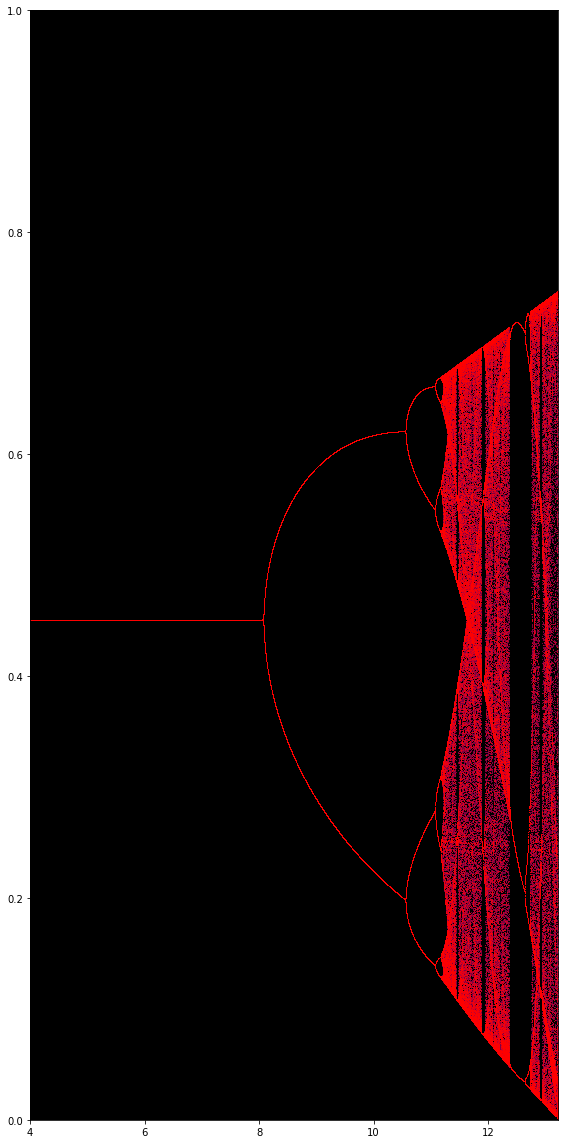}
  \caption{Bifurcation diagram for $\econ$ with $\step\in [4, \frac{4}{0.55^2}]$}
  \label{fig: bif-ppi045}
\end{subfigure}%
\begin{subfigure}{.5\textwidth}
  \centering
  \includegraphics[width=.9\linewidth]{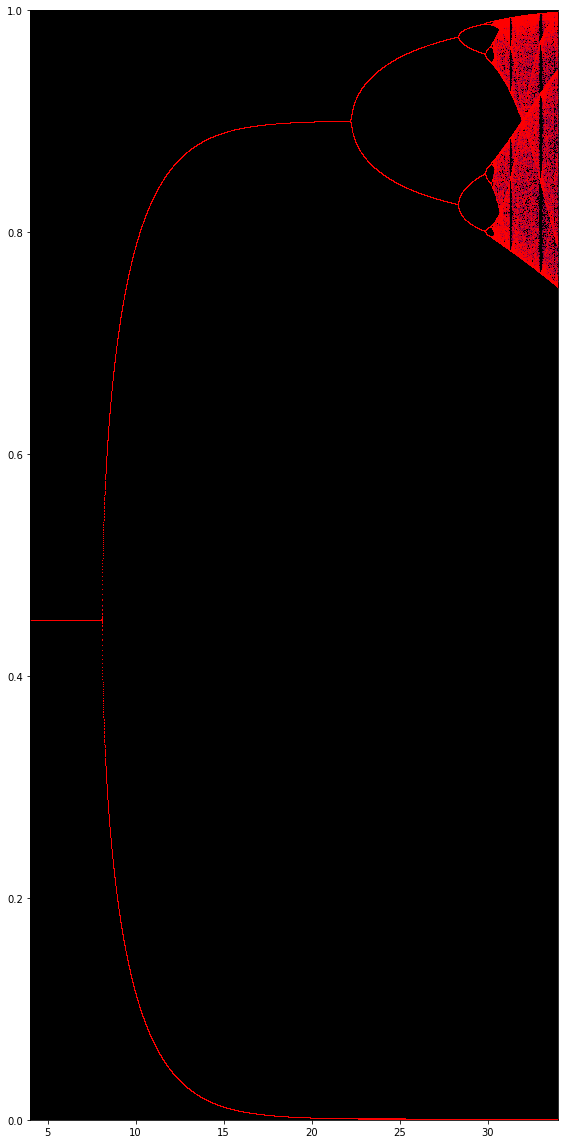}
  \caption{Bifurcation diagram for $\learn$ with  $\step\in [4, 34]$} 
  \label{fig: bif-ew045}
\end{subfigure}
\caption{Bifurcation diagrams for $\econ$ and $\learn$ with the equilibrium $\eq=0.45$. The first bifurcation on both diagrams is at $\step_0=2/0.45\cdot 0.55=\frac{22}{9}$. But then the second bifurcation for $\econ$ is much faster (around $\step=10.5$ while for $\fcs$ it is around $\step=25$). It is worth pointing out that period 3 can be detected for $\econ$ at $\step=12.5$, while for $\fcs$ it is $\step\approx 55.5$.}
\label{fig:test}
\end{figure}

Theorem \ref{thm:chaos-game} shows that the dynamics given by the models we analyze, described by \eqref{eq:RD}  in the continuous case, exhibit qualitatively different behaviors in the discrete case.
This points to a distinct nature of micro foundations of applications, and interpretations, of evolutionary game theory in different scientific contexts. First, it puts its biological foundations and foundations in economics and computer science at odds. In the biological model we see convergence to the unique Nash equilibrium regardless of the length of the evolutionary cycle. On the contrary, for \cref{mod:econ} and \cref{mod:learn} convergence to Nash equilibrium depends on the revision rate and learning rate, respectively. The convergence result fails for these models after crossing the value $\step_0=2/{\gain_A\gain_B}$, when both systems become unpredictable.
Nevertheless,   \cref{mod:econ} and \cref{mod:learn}  differ in other aspects. For   \cref{mod:econ} once the Nash equilibrium looses stability (which happens only if $\gain_A+\gain_B<3\gain_A<2(\gain_A+\gain_B)$), the system is going through period doubling bifurcation relatively fast, and for values of $\gain_A$ and $\gain_B$ close to each other ($\eq$ close to $1/2$), leads to periodic orbits of any period and chaotic behavior (see Figures \ref{fig:test}, \ref{fig: attractingper}). On the other hand,   increasing learning rate in \cref{mod:learn} will inevitably lead to chaotic behavior if only $\gain_a\neq \gain_B$, but the speed of period doubling road to chaos depends on the asymmetry of gains $\gain_A$ and $\gain_B$. In particular, for values of $\gain_A$, $\gain_B$ for which  in \cref{mod:econ} we detect chaos, in  \cref{mod:learn} the period doubling road to chaos can be (extremely) slow (compare bifurcation diagrams in Figure \ref{fig:test}).  

\subsection{Skeleton of the proof of Theorem \ref{thm:chaos-game}}
For the game \eqref{eq:game} we can concentrate on the population share of $A$-strategists.
Then 
\begin{itemize} 
\item for  \cref{mod:bio} we get the map $\bio \colon I\mapsto I$, of the unit interval $I=[0,1]$, given by
\begin{equation}
\label{biomap}\bio (\state)=\state \frac{1 + (\gain_A+\gain_B) \step \eq (1-\state)}{1 + (\gain_A+\gain_B)\step \state (1-\state)}.
\end{equation}
\item for \cref{mod:econ} we are interested in the dynamics of  the map $\econ \colon I\mapsto I$ given by
\begin{equation} \label{econmap} \econ (\state)=\state(1-(\gain_A+\gain_B)\step(1-\state)(\state-\eq)).\end{equation}
\item for  \cref{mod:learn} we get $\learn \colon I\mapsto I$ given by
\begin{equation} \label{csmap} \learn(\state)=\frac{\state}{\state+(1-\state)\exp[(\gain_A+\gain_B)\step (\state-\eq)]}.\end{equation}
\end{itemize}
Thus, in the proof of Theorem \ref{thm:chaos-game} we will focus on dynamics of these maps and how it depends on the choice of the step size $\step$. Without loss of generality we can assume that $\gain_A+\gain_B=1$, as otherwise we can proceed for the step size $\step'=(\gain_A+\gain_B)\step$. We begin in Section \ref{sec:global} with the discussion on the conditions for global convergence for these maps. To this aim
 we introduce Schwarzian derivative
 \[
 Sf \equiv \frac{f'''}{f'} - \frac{3}{2} \left( \frac{f''}{f'} \right)^2,
\]
as negative Schwarzian derivative guarantees \emph{good} behavior of the interval map.
 We will show that under conditions, which all studied maps 
 fulfill,  existence of an interior attracting fixed point will imply global convergence of dynamics to this  point (see Proposition \ref{globalppi}). 
 In Section \ref{sec:gdmod1} we use Proposition \ref{globalppi} to show global convergence of game dynamics introduced by the map \eqref{biomap} for any $\step$.
 Dynamics of $\econ$ is studied in Section \ref{sec:gdmod2}. By similar argument as for \cref{mod:bio} we show that as long as $\eq$ is attracting it attracts all points. Nevertheless, once the step size crosses the value of $2/\eq(1-\eq)$ the system becomes unstable and for a range of  values of $\eq$ close to $1/2$ will eventually (for large $\step$) be chaotic. To show that we carefully choose a point $x_0$, which fulfills assumptions of Li-Misiurewicz-Panigiani-Yorke theorem  \cite{li1982odd}. 
Finally, in Section \ref{sec:gdmod3} we discuss dynamics of \eqref{csmap}. Once more Proposition \ref{globalppi} guarantees that  $\eq$ attracts all points from $(0,1)$ as long as the step size is smaller than $2/\eq(1-\eq)$. The map $\learn$ is already known in the literature, see   \cite{Thip18,CFMP2019,palaiopanos2017multiplicative}, and the complete proof of the rest of the theorem comes from Theorems 3.10 and 3.11 from \cite{Thip18}.  Here, for the completeness of the presentation, we sketch the proof of Li-Yorke chaos for $\eq\neq 1/2$.


\subsection{Auxiliary result}
\label{sec:global}
Let \[\mathcal{F}=\{f\colon I\mapsto I, \; \text{Fix}f=\{0,\eq,1\},\;  0,1 \;\text{repelling fixed points}, \eq\in (0,1)\},\]
where $\text{Fix}f$ denotes the set of fixed points of $f$.

\begin{lemma} \label{attr} 
Let $f\in \mathcal{F}$.
If the trajectories of all points $0<\state<\eq$ are attracted to $\eq$, then the trajectories of all points from $(0,1)$ are attracted to $\eq$. Similarly, if the trajectories of all points $1>\state>\eq$ are attracted to $\eq$, then the trajectories of all points from $(0,1)$ are attracted to $\eq$.
\end{lemma}

\begin{proof}
Assume that there is a point in $(0,1)$, whose trajectory is not attracted to $\eq$. Since both 0 and 1 are repelling, by \cite{SKSF}, $f$ has a periodic orbit of period 2. If the trajectories of all points $\state<\eq$ (respectively, $\state>\eq$) are attracted to $\eq$, this periodic orbit has to lie entirely to the right (respectively, left) of $\eq$. Thus, there is a fixed point to the right (respectively, left) of $\eq$, a contradiction.
\end{proof}

\begin{proposition} \label{globalppi} 
Let $f\in \mathcal{F}$. 
Let $f$ fulfill one of the following conditions
\begin{enumerate}
\item $f$ is increasing  
\item $f$ is bimodal and its Schwarzian derivative is negative.
\end{enumerate}
If $\eq$ is attracting, then it is globally attracting.
\end{proposition}

\begin{proof}
If $f$ is strictly increasing, then it does not have a periodic orbit of period 2, so $\eq$ is globally attracting.

Assume that $f$ is bimodal. If $\eq$ belongs to the left or right lap, then, by Lemma \ref{attr}, $\eq$ is globally attracting. Assume that $\eq$ belongs to the interior of the middle lap. Because the Schwarzian derivative of $f$ is negative, then, by Singer's theorem \cite{de2012one}, the interval joining $\eq$ with one of the critical points of $f$ is in the basin of attraction $A$ of $\eq$. We may assume that this critical point is the left one, $\stratcl$. There is a unique point $\straty<\stratcl$ such that $f(\straty)=\eq$. Then $f([\straty,\stratcl])=f([\stratcl,\eq])\subset A$, so $[\straty,\eq]\subset A$. For every point $\state<\straty$ we have $\state<f(\state)<\eq$. Therefore, the trajectory of $\state$ increases as long as it stays to the left of $\straty$. Since there are no fixed points to the left of $\straty$, the trajectory has to enter $[\straty,\eq]$ sooner or later. This proves that $(0,\eq]\subset A$, so by Lemma \ref{attr},  $\eq$ is globally attracting.
\end{proof}

\subsection{Game dynamics under \cref{mod:bio}} \label{sec:gdmod1}
We study the dynamics introduced by \eqref{biomap}. First,  $\fbio$ is well-defined for every $\step>0$, that is $0\leq \bio(\state)\leq 1$ for every $\state\in [0,1]$.
Obviously values $\bio(\state)$ are always nonnegative and  the condition $\bio(\state)\leq 1$ is equivalent to $1+\step \state(1-\eq)\geq 0$, which is always satisfied. Thus, $\bio \colon [0,1]\mapsto [0,1]$ for any $\step>0$.

To study dynamics introduced by $\bio$ we look at fixed points and the derivative of $\bio$. So, $\bio(\state)=\state$ if and only if $\state=0$ or $\step \eq (1-\state)=\step \state(1-\state)$. Thus, our map has three fixed points: 0, 1 and $\eq$.
Now, we look at the derivative of $\bio$, which is equal to
\begin{equation}
\label{bioderivative} (\bio)'(\state)=\frac{\step \state^2-2\eq\step \state+\step \eq+1}{(1+\step \state(1-\state))^2}.
\end{equation}
We check stability of fixed points:
\[(\bio)'(0)=1+\step \gain_A>0,\;\;(\bio)'(1)=1+\step \gain_B>0\]
for any $\step>0$. So, both 0 and 1 are repelling. Thus, $f\in\mathcal{F}$. Moreover, \[\bio'(\eq)=\frac{1}{1+\step\gain_A\gain_B},\]
so $|(\bio)'(\eq)|<1$ and $\eq$ is attracting for any value of $\step>0$. Finally, we look at the monotonicity of $\bio$.
The sign of the derivative of $\bio$ depends only on the sign of  \[Q(\state)=\step \state^2-2\eq \step \state+\step \eq+1.\] Nevertheless, 
\[\Delta=4\step (\step \eq(\eq-1)-1)\] is always negative as $\eq\in (0,1)$ and $\step>0$. So, $\bio$ has no extrema and, as $(\bio)'(0)>0$, 
 $\bio$ is an increasing map for any $\step$.

Therefore, by Proposition \ref{globalppi} we obtain the part of Theorem \ref{thm:chaos-game} on game dynamics under  \cref{mod:bio}.\footnote{Thus, if only the initial state of the population is polymorphic, then the system converges to the evolutionary stable state $\eq$. On the other hand, as 0 and 1 are repelling, monomorphic populations are sensitive to small perturbations.}

\subsection{Game dynamics under \cref{mod:econ}} \label{sec:gdmod2}
Game dynamics for the map $\econ$ is well-defined only when we cannot leave the simplex. Thus, we have to assume that
\[\step\leq \min_{\state\in (\eq,1)}\frac{1}{(1-\state)(\state-\eq)}=\frac{4}{(1-\eq)^2}\;\;
\text{and}\;\; 
\step\leq \min_{\state\in (0,\eq)}\frac{1}{\state(\eq-\state)}=\frac{4}{\eq^2}.\]
Thus, $\fecon$ is well-defined when
\[\step\leq \step^*=\min\left\{\frac{4}{\eq^2},\frac{4}{(1-\eq)^2}\right\}.\]

Fixed points of $\fecon$ are 0 and solutions of the equation \[\step (1-\state)(\state-\eq)=0.\] So the map $\fecon$ has three fixed points: $0$, $1$ and $\eq\in (0,1)$.

As 
\begin{equation} \label{fder} (\fecon)'(\state)=3\step \state^2-2\step (1+\eq)\state+\eq\step +1\end{equation}
we see that $(\fecon)'(0)=1+\eq\step$ and $(\fecon)'(1)=1+(1-\eq)\step$ so both 0 and 1 are repelling for every $\step$. Thus, $f\in\mathcal{F}$.
Let's look at stability of the interior fixed point $\eq$. This point will be attracting as long as $|(\fecon)'(\eq)|<1$. We have
\[(\fecon)'(\eq)=1-\step \eq(1-\eq),\]
so $\eq$ is attracting as long as $\step<\frac{2}{\eq(1-\eq)}$ and repelling otherwise.

From \eqref{fder} our map is increasing as long as $\step<\frac{3}{1-\eq+\eq^2}$. Then, for $\step>\frac{3}{1-\eq+\eq^2}$ the map $\fecon$ is bimodal with critical points
\[\stratcl=\frac{1+\eq}{3}-\frac 13 \sqrt{\frac{\step(1+\eq)^2-3\eq\step-3}{\step}},\;\;\;\text{and}\;\;\;\stratcr=\frac{1+\eq}{3}+\frac 13 \sqrt{\frac{\step(1+\eq)^2-3\eq\step-3}{\step}}.\] 

For interval maps with \emph{good} behavior are those with negative Schwarzian derivative. We should not expect that $S\fecon<0$ for $\step<\frac{3}{1-\eq+\eq^2}$ as then $\econ$ is a homeomorphism. For $\step>\frac{3}{1-\eq+\eq^2}$ we have the following fact.
\begin{lemma} \label{ls}
If $\step>\frac{3}{1-\eq+\eq^2}$, then $S\fecon<0$.
\end{lemma}

\begin{proof}
Elementary calculations give us 
\[(\fecon)'(\state)=3\step \state^2-2\step(1+\eq)\state+1+\eq\step,\]
\[(\fecon)''(\state)=6\step \state-2\step (1+\eq),\]
\[(\fecon)'''(\state)=6\step.\]
Schwarzian derivative is negative if and only if $2(\fecon)'(\fecon)'''-3((\fecon)'')^2<0$, thus we want to show that
\[12\step [3\step \state^2-2\step(1+\eq)\state+1+\eq\step]-3[6\step \state-2\step(1+\eq)]^2<0.\]
Thus, \[\step (1+\eq+\eq^2-4(1+\eq)\state+6\state^2)>1\]
for any $\state$. The minimum of $g(\state)=1+\eq+\eq^2-4(1+\eq)\state+6\state^2$ is at $\state=\frac{1+\eq}{3}$ and is equal to $\frac{\eq^2-\eq+1}{3}$.
Therefore, $S\fecon<0$ for $\step >\frac{3}{1-\eq+\eq^2}$.
\end{proof}

Now we are able to describe what happens when $\eq$ is attracting, that is when $\step<2/\eq(1-\eq)$.

By Lemma \ref{ls} and Proposition \ref{globalppi} we get that $\eq$ attracts all trajectories as long as $\step\leq 2/\eq(1-\eq)$. 
In particular, as long as $\step^*< 2/\eq(1-\eq)$ Nash equilibrium always attracts all points. This implies first two results for  \cref{mod:econ}.

\begin{figure}
\centering
\includegraphics[width=0.9\textwidth]{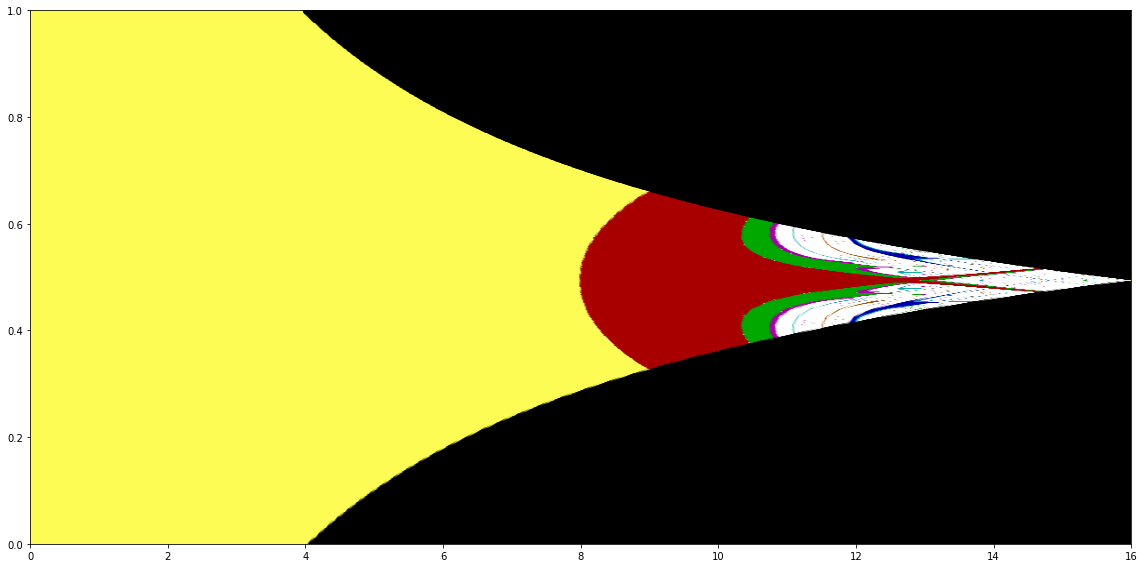} 

\caption{Period diagrams of the small-period attracting periodic orbits associated with the map $\fecon$ (drawn on the black background). The horizontal axes are $\step \in [4,16]$ and the vertical axes are the asymmetry of cost $\eq \in [0,1]$. The colors encode the periods of attracting periodic orbits as follows: period 1 (fixed point, which is Nash equilibrium $\eq$) = {\color{yellow}yellow}, period 2 = {\color{red}red}, period 3 = {\color{blue}blue}, period 4 = {\color{green}green}, period 5 = {\color{brown}brown}, period 6 = {\color{cyan}cyan}, period 7 = {\color{darkgray}darkgray}, period 8 = {\color{magenta}magenta}, and period larger than 8 = white. 
The picture is generated from the following algorithm: 20000 preliminary iterations are discarded. Then a point is considered periodic of period $n$ if $|(\fecon)^n(\state)-\state|<10^{-10}$ and it is not periodic of any period smaller than $n$. As long as we are in the yellow region we have convergence to Nash equilibrium, once we get out of this region almost all trajectories will never converge to the fixed point.}
\label{fig: attractingper}
\end{figure}
\begin{figure}
\centering
\includegraphics[width=0.9\textwidth]{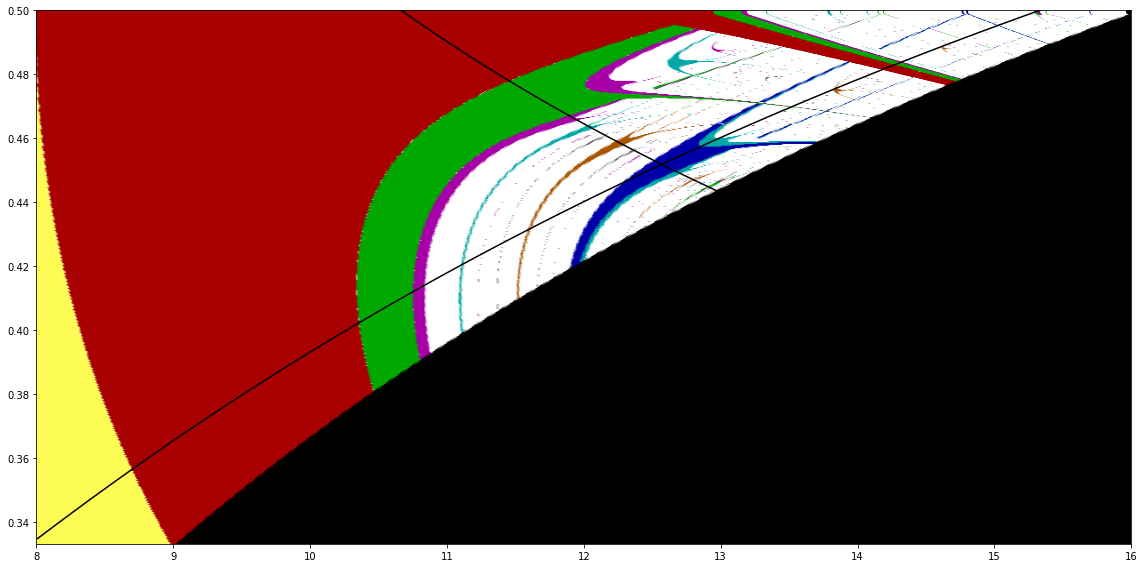} 

\caption{Period diagrams of the small-period attracting periodic orbits associated with the map $\fecon$ (drawn on the black background). The horizontal axes are $\step \in [8,16]$ and the vertical axes are the asymmetry of cost $\eq \in [1/3,1/2]$. The colors encode the periods of attracting periodic orbits as follows: period 1 (fixed point) = {\color{yellow}yellow}, period 2 = {\color{red}red}, period 3 = {\color{blue}blue}, period 4 = {\color{green}green}, period 5 = {\color{brown}brown}, period 6 = {\color{cyan}cyan}, period 7 = {\color{darkgray}darkgray}, period 8 = {\color{magenta}magenta}, and period larger than 8 = white. 
The picture is generated from the following algorithm: 20000 preliminary iterations are discarded. Then a point is considered periodic of period $n$ if $|(\fecon)^n(\state)-\state|<10^{-10}$ and it is not periodic of any period smaller than $n$. On the picture we also draw the black curves for conditions \eqref{f1b6} and \eqref{f226}.}
\label{fig: attractingperzoom}
\end{figure}


Our aim is to investigate the long-term behavior of the orbits of
$\fecon$. As for $\step>2/\eq(1-\eq)$ all fixed points are repelling there is no convergence to the equilibrium, and we need to study this case more thoroughly. In particular, we should look for periodic
orbits, their stability, or chaos. Figure \ref{fig: attractingper} suggests existence of chaotic behavior for $\eq$ near-uniform and large values of $\step$.  When speaking of chaos, we will use
its most popular kind, \emph{Li-Yorke chaos}.\footnote{For description of Li-Yorke chaos and its connections with other definitions of chaos we refer the reader to \cite{Blanchard,Ruette}}

As we deal with the interval map to prove Li-Yorke chaos we can use odd period argument. Thus, we will show the following proposition. 
\begin{proposition}
If $\eq\in \left(\frac{1}{23}(31-12\sqrt{3}),\frac{1}{23}(12\sqrt{3}-8)\right)$, then there exists a unique $\step_{\eq}^{II}$ such that  $\fecon$ has periodic point of period 3 for any  $\step\in (\step_{\eq}^{II},\step^*]$.
\end{proposition}

\begin{proof}

We assume that $\eq\in (0,1/2]$. We will show that for $\eq$ sufficiently close to $1/2$ the following conditions hold:

\begin{equation}
\label{f1b6} \fecon\left(\frac{1+\eq}{6}\right)>\frac{1+\eq}{2}
\end{equation}

and

\begin{equation}
\label{f226} (\fecon)^2\left(\frac{1+\eq}{2}\right)<\frac{1+\eq}{6}.
\end{equation}

Before justifying these inequalities let us show how \eqref{f1b6} and \eqref{f226} guarantee existence of the periodic point of period 3.
From \eqref{f1b6}, continuity of $\fecon$ and the fact that $\fecon(0)=0$, $\fecon(\eq)=\eq$ we obtain existence of $\state_0\in (c_l,\eq)$ such that $\fecon(\state_0)=\frac{1+\eq}{2}$.
From \eqref{f226}  and the fact that $\fecon(\state)>\state$ when $\state<\eq$, we have that
\[(\fecon)^3(\state_0)=(\fecon)^2(\frac{1+\eq}{2})<\frac{1+\eq}{6}<c_l<\state_0<\fecon(\state_0).\]

 Therefore, from theorem by Li, Misiurewicz, Panigiani and Yorke \cite{li1982odd},  $\fecon$ has periodic orbit of period 3. By $f_{\step,\eq}^{II}(\state)=1-f_{\step,1-\eq}^{II}(\state)$ we will conclude similar result for $\eq>1/2$.\footnote{Although from numerical experiments we may conclude that periodic orbits of period 3 may arise for $\step$ smaller than those for which conditions \eqref{f1b6} and \eqref{f226} hold, it seems that it estimates presence of period 3 quite well, see Figure \ref{fig: attractingperzoom}.}
 
 Now we will show conditions \eqref{f1b6} and \eqref{f226}. We begin with the first inequality, which is equivalent to 
\[-\frac{5}{216}\step(1+\eq)^3+\frac 16(1+\eq)(1+\step \eq)>\frac{1+\eq}{2}.\]
Thus,
\[\step\left(-\frac{5}{36}(1+\eq)^2+\eq\right)>2.\]
There is no positive $\step$ which fulfills this condition as long as $\eq\leq \frac 15$. If $\eq>\frac 15$, then $\step>\frac{72}{-5\eq^2+26\eq-5}$. As $\step\leq \step^*$ we obtain that \eqref{f1b6} holds when 
\begin{equation} \label{bbb} \eq\in\left(\frac{1}{23}(31-12\sqrt{3}),\frac 12\right].\end{equation}

Next we turn our attention to condition \eqref{f226}, which is equivalent to 
\begin{equation} \label{delta01} \frac{3}{256}(4-\step(1-\eq)^2)(64-16\step(1-\eq)^2+\step^3(1-\eq)^4(1+\eq)^2)<1.\end{equation}
Obviously if $\step=\step^*$, then the left hand side of \eqref{delta01} is equal to zero. 
Define
\begin{equation} \label{largef} F(\step)=\frac{3}{256}(4-\step(1-\eq)^2)(64-16\step(1-\eq)^2+\step^3(1-\eq)^4(1+\eq)^2).\end{equation}
We want to  show that there is a unique  $\step_{\eq}^{II}$ such that for $\step> \step_{\eq}^{II}$ values of $F$ are smaller than 1.

We know that $F(0)=3$ and $F(\step^*)=0$. 
The derivative of $F$ is equal to
\[F'(\step)=\frac{3}{256}\left(-68(1-\eq)^2+2\step(1-\eq)^4+12\step^2(1-\eq)^4(1+\eq)^2-4\step^3(1-\eq)^6(1+\eq)^2\right).\]
As $F'$ is a polynomial of degree $3$ and \[F'(0)=-\frac{51}{64(1-\eq)^2}<0,\] with \[F'(\step^*)=\frac{3}{256}(-60(1-\eq)^2-64(1+\eq)^2)<0,\] we get that inside the interval $[0,\step^*]$  the map $F'$ can have 0 or 2 roots. As \[F'\left(\frac{\step^*}{2}\right)=-\frac{3}{16}(3-10\eq+3\eq^2)>0\] when $\eq\in (\frac 13,\frac 12]$, we exclude possibility of no roots. 
Thus, inside the interval $[0,\step^*]$ the map $F$ has local minimum $\step_{\min}$ and local maximum $\step_{\max}$, $\step_{\min}<\step_{\max}$.
As $F(\step^*)=0$ we know that there exists $\step_{\eq}^{II}$ such that  \eqref{f226} holds for $\step> \step_{\eq}^{II}$. To show uniqueness of $\step_{\eq}^{II}$ we need to show that $F(\step_{\min})>1$.

We have
\[F''(\step)=\frac{3}{128}(1-\eq)^4(1+12\step(1+\eq)^2-6\step^2(1+\eq)^2(1-\eq)^2),\]
and simple calculations show that $F''$ has one positive root $\step_1$. Thus,
\[0<\step_{\min}<\step_1<\step_{\max}\]
and $F$ is convex in $(0,\step_1)$.

Now take $\delta= \frac{\step^*}{4}$. Then $F'(\frac{\step^*}{4})=-\frac{3}{32}(11\eq^2-26\eq+11)<0$ as $\eq<\frac 12<\frac{1}{11}(13-4\sqrt{3})$. 
Define an affine map
\[G(\step)=F'\left(\frac{\step^*}{4}\right)\cdot \step+F\left(\frac{\step^*}{4}\right)-F'\left(\frac{\step^*}{4}\right)\cdot \frac{\step^*}{4}.\]

Because $F'(\frac{\step^*}{4})<0$, the map $G$ is decreasing. Convexity of $F$ guarantees that $F(\step_{\min})\geq G(\step_{\min})$.
Moreover, 
$\frac{\step^*}{4}<\step_{\min}<\frac{\step^*}{2}$.
Thus, we have
\[F(\step_{\min})\geq G(\step_{\min})>G\left(\frac{\step^*}{2}\right),\]
where
\[G\left(\frac{\step^*}{2}\right)=F\left(\frac{\step^*}{4}\right)+\frac{\step^*}{4}F'\left(\frac{\step^*}{4}\right).\]

Let
\[H(\eq)=G\left(\frac{\step^*}{2}\right)-1=\frac{-79\eq^2+290\eq-79}{256(1-\eq)^2}.\]

We have $H(\eq)>0$ when $\eq\in (\frac{1}{79}(145-8\sqrt{231}),\frac{1}{79}(145+8\sqrt{231}))$. Combining all restrictions we obtain that

\[F(\step_{\min})\geq G(\step_{\min})>G\left(\frac{\step^*}{2}\right)>1\] 
for $\eq\in \left(\frac{1}{23}(31-12\sqrt{3}),\frac 12\right]$. This completes the proof for $\eq\in \left(\frac{1}{23}(31-12\sqrt{3}),\frac 12\right]$. 

Finally,
\begin{equation}\label{symmecon}1-f_{\step,1-\eq}^{II}(1-\state)=1-(1-\state)(1-\step \state(\eq-\state))=\state(1+\step(\eq-\state)-\step \state(\eq-\state))=f_{\step,\eq}^{II}(\state),\end{equation}
which implies the assertion of the theorem for $\eq\in (\frac{1}{23}(31-12\sqrt{3}),\frac{1}{23}(12\sqrt{3}-8))$.
\end{proof}

By Sharkovsky theorem \cite{SKSF} 
we obtain the assertion of the theorem.

\subsection{Dynamics of $\fcs$} \label{sec:gdmod3} The map \eqref{csmap} is already known and its dynamics was studied in various papers \cite{Thip18,CFMP2019,palaiopanos2017multiplicative}. 
For the completeness of the exposition we describe here crucial properties of this map. 
Fixed points of $\fcs$ are $0$ and the roots of the equation
\[(1-\state)(1-\exp (\step (\state-\eq)))=0.\] 
So fixed points of $\fcs$ are 0, 1 and $\eq$.
The derivative of $\fcs$ is given by
\begin{equation}\label{der}
(\fcs)'(\state)=\frac{(\step \state^2-\step \state+1)\exp(\step(\state-\eq))}
{\big(\state+(1-\state)\exp(\step(\state-\eq))\big)^2}.
\end{equation}
Thus,
\[
(\fcs)'(0)=\exp(\step \eq),\ \ (\fcs)'(1)=\exp(\step(1-\eq)),\ \ (\fcs)'(\eq)=\step \eq^2- \step \eq+1.
\]
We see that the fixed points 0 and 1 are always repelling, while $\eq$
is repelling if $\step>\frac 2{\eq(1-\eq)}$.

The critical points of $\fcs$ are solutions to $\step \state^2-\step \state+1=0$. Thus,
if $0<\step\le 4$, then $\fcs$ is strictly increasing. If $\step>4$, it has
two critical points
\begin{equation}\label{crit}
\kappa_l=\frac12-\sqrt{\frac14-\frac{1}{\step}},\ \ \
\kappa_r=\frac12+\sqrt{\frac14-\frac{1}{\step}},
\end{equation}
so the map $\fcs$ is bimodal.


Let us investigate regularity of $\fcs$.  By Proposition 3.2 from \cite{Thip18} the map $\fcs$ has negative Schwarzian derivative for $\step>4$.

For maps with negative Schwarzian derivative each attracting or
neutral periodic orbit has a critical point in its immediate basin of
attraction. Thus, we know that if $\step>4$ then $\fcs$ can have at
most two attracting or neutral periodic orbits.

Thus, $\fcs\in \mathcal{F}$, so from Proposition \ref{globalppi} we get that as long as $\step<2\eq/(1-\eq)$, Nash equilibrium attracts all trajectories from $(0,1)$.

Now, we sketch the proof of Li-Yorke chaos for $\fcs$ when $\gain_A\neq \gain_B$. Although this fact was already shown in \cite{Thip18}, we present here this proof to give a proper comparison with the proof of chaotic behavior for pairwise proportional imitation. To show chaotic behavior one can use period 3 arguments once more. To do that we notice that $\fcs(\state)>\state$ when $\state<\eq$. Moreover,
\[(\fcs)^n(\state)=\frac{\state}{\state+(1-\state)\exp (\step\sum_{k=0}^{n-1}((\fcs)^k(\state)-\eq))}.\]
So, $(\fcs)^3(\state)<\state$ if and only if \begin{equation} \label{fcscondper3}\state+\fcs(\state)+(\fcs)^2(\state)>3\eq.\end{equation} For $0<\eq<1/2$ we take $3\eq-1<x<\eq$. As $\lim\limits_{\step\to\infty}\fcs(\state)=1$ and $(\fcs)^2(\state)>0$ we  see that \eqref{fcscondper3} will be fulfilled for sufficiently large $\step$. Thus, we can use Misiurewicz-Li-Panigiani-Yorke theorem and Sharkovsky theorem and obtain Li-Yorke chaos for $\eq<1/2$. As \begin{equation}\label{sym} \phi\circ f_{\step,\eq}^{III}=f_{\step,1-\eq}^{III}\circ\phi,\ \ \textrm{where}\ \ \phi(\state)=1-\state. \end{equation}
we have the same result for $\eq>1/2$.

The result for $\gain_A=\gain_B$ follows from Theorem 3.10 from \cite{Thip18}. This completes the proof of Theorem  \ref{thm:chaos-game}. 


Finally, let us comment on the technical differences in the proof for  \cref{mod:econ} and  \cref{mod:learn}. Both are done by using  Li-Misiurewicz-Panigiani-Yorke theorem \cite{li1982odd} and period 3 argument. To show this we need to find a point $\state_0\in (0,1)$ such that $f^3(\state_0)<\state_0<f(\state_0)$ or $f^3(\state_0)>\state_0>f(\state_0)$. Nevertheless, in \cref{mod:learn} the trick is to take very large step size (depending on $b$). Moreover, when $\step$ is sufficiently large, the choice of a point which fulfills the condition is easy as we can take it from the large range of values, see Figure \ref{fig:fcs35}. 
This can't be done in \cref{mod:econ}, where the step size is bounded. In addition, the condition for the third iterate is met only on small interval which varies with the step size, see Figure 
\ref{fig:fecon5}. Thus, we carefully choose then a point which first iterate lays close to the (right) critical point, but values can be estimated analytically. This cannot be done for the $\fcs$ map, where we can get only numerical estimations. 

\begin{figure}
\centering
\begin{subfigure}{.48\textwidth}
  \centering
  \includegraphics[width=.9\linewidth]{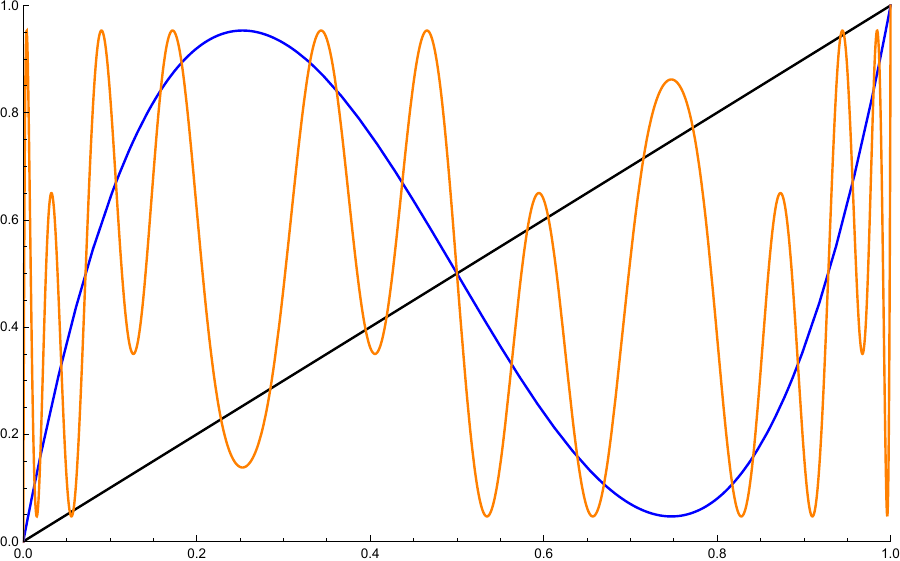}
\end{subfigure}%
\begin{subfigure}{.48\textwidth}
  \centering
  \includegraphics[width=.9\linewidth]{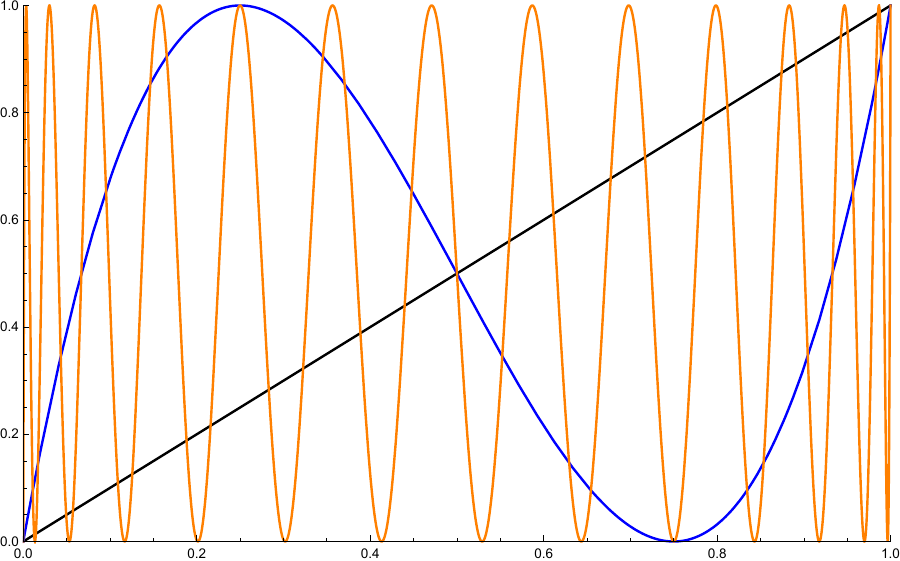}
\end{subfigure}
\caption{Map $\fecon$ and its third iterate when $\eq=0.5$ with $\step=15$ (left) and $\step=\step^*=16$ (right).}
\label{fig:fecon5}
\end{figure}

\begin{figure}
\centering
\begin{subfigure}{.48\textwidth}
  \centering
  \includegraphics[width=.9\linewidth]{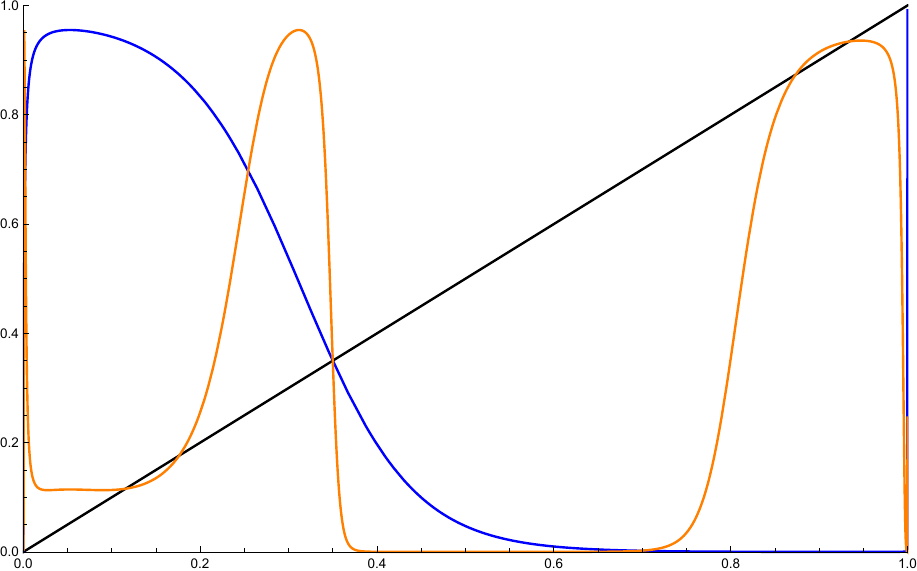}
\end{subfigure}%
\begin{subfigure}{.48\textwidth}
  \centering
  \includegraphics[width=.9\linewidth]{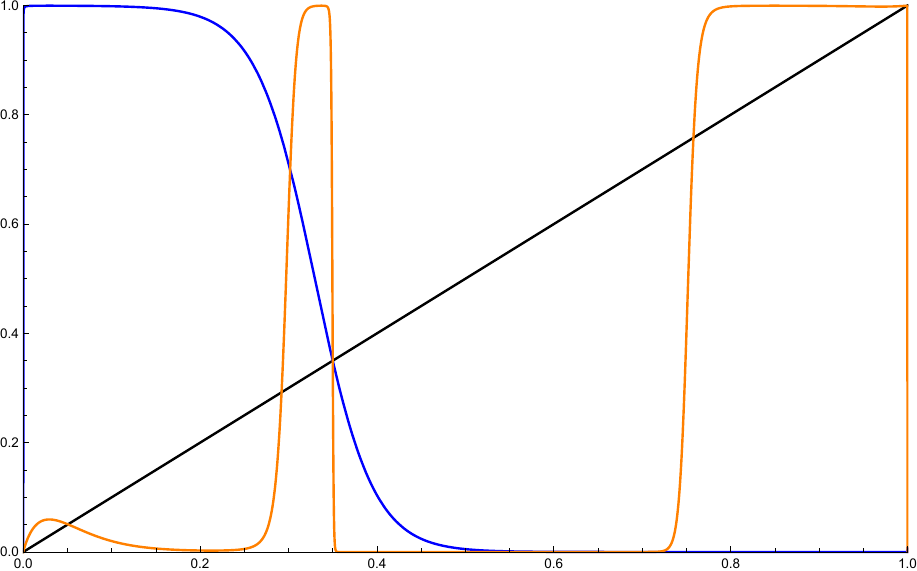}
\end{subfigure}
\caption{Map $\fcs$ and its third iterate when $\eq=0.35$ with $\step=20$ (left) and $\step=35$ (right).}
\label{fig:fcs35}
\end{figure}


\section{Discussion}
\label{sec:discussion}

Our work shows that three distinct game dynamics in discrete time \textendash\ a biological model of intra-species competition in the spirit of \citet{Mor62}, evolution under the \acl{PPI} protocol of \citet{Hel92} in economic theory, and learning with the \acl{EW} algorithm from the theory of adversarial online learning \citep{ACBFS95} \textendash\ exhibit qualitatively different long-run properties, despite the fact that they all share the \emph{same} continuous-time limit \textendash\ the \acl{RD}.
This disconnection occurs even in the simplest of games \textendash\ a $2\times2$ symmetric random matching congestion game \textendash\ and leads to drastically different predictions (or lack thereof):
\begin{enumerate*}
[\itshape a\upshape)]
\item
the biological model guarantees universal convergence to \acl{NE} for all initial conditions and all equilibrium and (hyper)parameter configurations;
\item
the economic model demonstrates the entire range of possible behaviors for large $\step$ (convergence, instability, periodic and chaotic behavior);
and, finally,
\item
fully chaotic behavior under the \ac{EW} algorithm (for large $\step$).
\end{enumerate*}
This divergence of behaviors provides a crisp cautionary tale to the efect that ``\emph{discretization matters}'', and serves to highlight the extent to which concrete conclusions can be drawn from the behavior of continuous-time models \textendash\ or, rather, the failure thereof.

\bibliographystyle{abbrvnat} 
\bibliography{IEEEabrv,bibtex/Bibliography-PM,bibtex/ms}

\end{document}